\documentclass[12pt,fleqn]{amsart} 

\usepackage{amsthm}
\usepackage{amsfonts}
\usepackage[english]{babel}
\usepackage{graphicx}
\usepackage{soul}
\usepackage{stfloats}
\usepackage{morefloats}
\usepackage{cite}
\usepackage{lscape}
\usepackage{epstopdf}
\usepackage{braket}
\usepackage[lite]{amsrefs}
\usepackage{mathbbol}
\usepackage{tikz,tikz-cd}
\usepackage{shuffle}
\usepackage{bbding}

\usepackage{algorithm,algorithmicx,algpseudocode}

\usepackage{amsmath,amstext,amsopn,amsfonts,eucal,amssymb}
\usepackage{graphicx,wrapfig,url}

\newtheorem{theorem}{Theorem}[section]
\newtheorem{corollary}[theorem]{Corollary}
\newtheorem{lemma}[theorem]{Lemma}

\newtheorem{definition}[theorem]{Definition}

\newtheorem{remark}[theorem]{Remark}

\newtheorem{hypothesis}[theorem]{Hypothesis}

\begin{document}

	\title[Projection Methods]{Projection Methods for Operator Learning and Universal Approximation}

	\author{Emanuele Zappala} 
	\address{Department of Mathematics and Statistics, Idaho State University\\
		Physical Science Complex |  921 S. 8th Ave., Stop 8085 | Pocatello, ID 83209} 
	\email{emanuelezappala@isu.edu}

	\maketitle

	\begin{abstract}
		We obtain a new universal approximation theorem for continuous (possibly nonlinear) operators on arbitrary Banach spaces using the Leray-Schauder mapping. Moreover, we introduce and study a method for operator learning in Banach spaces $L^p$ of functions with multiple variables, based on orthogonal projections on polynomial bases. We derive a universal approximation result for operators where we learn a linear projection and a finite dimensional mapping under some additional assumptions. For the case of $p=2$, we give some sufficient conditions for the approximation results to hold. This article serves as the theoretical framework for a deep learning methodology in operator learning.
	\end{abstract}
	
	

	\date{\empty}

	\tableofcontents

	\section{Introduction}
	
	Operator learning is a branch of deep learning involved with approximating (potentially highly nonlinear) continuous operators between Banach spaces. The interest of operator learning lies in the fact that it allows to model complex phenomena, e.g. dynamical systems, whose underlying governing equations are not known \cite{DeepOnet,ANIE_NAT,NIDE,operator_learning_survey}. The study of operator learning was initiated by the theoretical work \cite{chen}, whose implementation was given in \cite{DeepOnet}. Since then, this field has expanded significantly both in its theoretical and applied scope to encompass a variety of architectures \cite{fanaskov2023spectral,yamazaki2025finite,wu2025po,liu2024render,ANIE_NAT,NIDE,nelsen2024operator,li2020multipole,li2025harnessing}
	
	Projection methods, e.g. Galerkin methods, are approaches for finding solutions of an operator equation by approximating this on prescribed subspaces through a projection \cite{KVZ,Fle}. After projecting the operator equation on a subspace, it is not necessarily true that this equation has a solution. When projected solutions exist, upon increasing the dimension of the subspaces one would want the solutions to converge to a solution of the original equation. This is not necessarily the case. The main question of projection methods is whether projected solutions exist, and converge to a solution of the original non-projected equation. 
	
	We can formulate the problem of operator learning in relation to projection methods as follows. We want to learn projections on (finite dimensional) subspaces, and a map between subspaces such that we can approximate a target operator between Banach spaces. This operator is assumed to satisfy an operator equation whose solutions model the data, as formulated in \cite{ANIE_NAT,NIDE} for the case of integral and integro-differential operators. Our problem is therefore two-fold. First, we want to learn an operator whose solutions of a corresponding operator equation model data. Second, we want to approximate the operator on a projected space, and solve the projected operator equation. A similar approach for a class of integral operators was employed in \cite{Spectral}. More general overviews of operator learning techniques in their theoretical and practical frameworks can be found in \cite{operator_learning_survey,boulle2024mathematical,jha2025theory}.
	
	In the present article we theoretically address these problems, and derive a methodology for operator learning based on learning projections and mappings between projected spaces. Here, we obtain a universal approximation theorem for operators between Banach spaces using the Leray-Schauder mappings.  
	
	To contextualize this work, we summarize the hypotheses, approximation capabilities, and main implementational considerations related to some operator learning frameworks that are closely related to the present work. An overview of these models is found in Table~\ref{tab:models}. 
		
		\begin{table}\label{tab:models}
				\resizebox{\textwidth}{!}{\begin{tabular}{||c|c|c|c|c|c||}
					\hline 
					Model & Spaces & Hypotheses & Approximation & Known Proj & Known Bases \\
					\hline
					Leray-Schauder & Banach & Cont operator & Univ & $\checkmark / \times$ & $\checkmark / \times$ \\
					\hline
					$L^p$ Proj & $L^p$ &  Cont operator & Univ & $\checkmark$ & $\checkmark$ \\
					\hline
					DeepONet \cite{DeepOnet} &  Banach to uniform &  Cont operator & Univ & $\times$ & $\times$ \\
					\hline 
					NIE \cite{ANIE_NAT} &  Compact open top &  Cont integral operator & Univ & $\times$ & $\times$ \\
					\hline 
					ANIE \cite{ANIE_NAT} &  Compact open top &  Cont integral operator & Univ & $\times$ & $\times$ \\
					\hline 
					Spectral NIE \cite{Spectral} &  H\"older &  Frech\'et integral operator & Univ & $\checkmark$ & $\checkmark$ \\
					\hline 
					Spectral NO\cite{fanaskov2023spectral} & ? &  ? & ? & $\checkmark$ & $\checkmark$ \\
					\hline 
					PO-CKAN \cite{wu2025po}& ? &  ? & ? & $\times$ & $\times$ \\
					\hline
					OPNO \cite{liu2024render} & H\"older to $L^2$ & Cont operator & Univ & $\times$ & $\checkmark$ \\
					\hline
					FEPINN \cite{yamazaki2025finite} &  H\"older to $L^2$ & PD operator & ? & $\times$ & $\checkmark$ \\
					\hline
					RFM \cite{nelsen2024operator} & H\"older to $L^2$  & PD operator & ? & $\times$ & $\checkmark$ \\
					\hline
					PCA-Net \cite{hesthaven2018non,bhattacharya2021model}& Hilbert space  & $\mu$-meas+ & Prob $>1-\delta$ & $\checkmark$ & $\times$ \\
					\hline
					CNO \cite{raonic2023convolutional}& H\"older to $L^p$  & PD operator+ & Univ & $\times$ & $\times$ \\
					\hline
					MGNO \cite{li2020multipole}& $L^2$ & ? & ? & $\times$ & $\checkmark$ \\
					\hline
					AMG \cite{li2025harnessing}& $L^2$ & linear integral operator & ? & $\times$ & $\checkmark$ \\
					\hline
				\end{tabular}}
				\caption{Summary of varoius neural operator architectures with their theoretical hypotheses of applicability, approximation capabilities, and construction of approximations for projection on finite dimensional subspaces}
		\end{table}
	
		For each model in Table~\ref{tab:models}, we list their domain and codomain, the theoretical hypotheses under which their approximation bounds are known to hold, the corresponding approximation rate, and whether the finite-dimensional reduction methodology used to approximate the operator relies on a known or learned projection and basis. Here, by known we mean that the projection and/or basis is analytically defined rather than learned. When a `+' symbol appears among the hypotheses, it indicates that additional technical assumptions are required and are omitted for brevity. The question marks in Table~\ref{tab:models} indicate entries for which results are not currently available in the literature. More in detail, a ``?'' in the ``Spaces'' column indicates that the corresponding model has not been studied with respect to a specific normed function space. A ``?''  in the ``Hypotheses'' column indicates that the conditions under which the associated neural operators admit theoretical guarantees are not known. Finally, a ``?'' in the ``Approximation'' column, indicates that approximation results for the corresponding model have not been established yet.
		
		The first two models listed, Leray-Schauder and $L^p$ proj, are introduced in the present article. These models are formulated in a general theoretical framework, but they also admit concrete realizations. In particular, the Leray-Schauder model has been implemented in \cite{leray_schauder}, where it was extended to include learnable projection and basis components. Both versions are reported in the table, and the corresponding approximation error bounds for the learnable case have also been analyzed in \cite{leray_schauder}, thereby complementing the theoretical analysis provided in the present work. The cited error bounds for PCA-Net are those found in \cite{lanthaler2023operator}. The Leray-Schauder model, as shown below, is a universal approximator for continuous operators between arbitrary Banach spaces. The implementation presented in \cite{leray_schauder} also demonstrates how to address the problem of selecting suitable projection bases and how to enhance model stability by allowing learnable Leray-Schauder projections. The $L^p$ proj model, while retaining universal approximation properties for a broad and practically relevant class of operators in $L^p$ spaces, provides a simpler structure that also permits the theoretical analysis of fixed-point problems (i.e., operator equations of the second kind) via projection methods. It shows that the one-dimensional version proposed in \cite{Spectral} can be extended to multivariate polynomial bases by adopting a modified theoretical framework, leading to milder assumptions and a broader range of applicability. The numerical experiments in \cite{Spectral} further confirm that the approach considered in the present work can achieve high accuracy and stability in interpolation problems. 
		
		In general, different basis representations play complementary roles across families of problems. Fourier bases are particularly effective for periodic systems with frequency localization, while wavelet bases are well suited for capturing discontinuities and localized features due to their multiresolution nature. Graph bases present natural advantages when geometric information from the domain needs to be embedded into the problem. Polynomial bases, by contrast, are advantageous in smooth, nonlocal settings|such as integral equation formulations|and remain applicable even in the absence of periodicity, though they are less effective for discontinuous phenomena.
		
		These observations, along with the theoretical guarantees shown in Table~\ref{tab:models}, provide criteria for choosing neural operator models depending on the problems considered in practice. For instance, when considering problems related to PDEs, a natural choice might fall upon FEPINN, RFM, CNO. When smoothness assumptions are well understood, spectral-based approaches such as Spectral NIE or Spectral NO might be preferred.

	The approach based on Leray-Schauder mappings implies that we need to find elements of the Banach spaces that approximate given compact subsets. An algorithm that concretely implements this approach should also obtain the points whose linear subspace will be used for the Leray-Schauder (nonlinear) projection. The concrete implementation of these results has been addressed in \cite{leray_schauder}, which practically implemented part of the theoretical framework given in the present article. However, operator learning in practice is often formulated in concrete spaces of functions. Therefore, a reformulation of the methodology in this setting would be of value. 
	
	To address the latter, we consider the more specific (and important) case of $L^p$ spaces. In this case, we show that given a set of orthogonal polynomials with respect to a quasi-inner product with some additional assumptions, we can find a triple of neural networks defining two projections and a map between the projected spaces that approximate the given operator with arbitrary precision. When $p=2$, and we are working with the Hilbert space $L^2$, we give some simple sufficient conditions for the aforementioned results. We then turn to the problem of approximating solutions of the projected operator equation. We provide in this case some sufficient conditions for the projection to admit solutions for each $n$, and such that the solutions converge to a solution of the operator equation when $n \longrightarrow \infty$.

	The framework developed in the present work relates to the approximation of continuous operarators between general Banach spaces. We notice that this extends the framework of DeepONet \cite{chen,DeepOnet}, as the latter treats the case of maps where the target Banach space is given by the uniform norm. Our approach is also developed more specifically for mappings between $L^p$ spaces. In the approaches considered in this article, both for general Banach spaces and $L^p$ spaces, we provide explicit and analytically defined projecting operators. 
	
	The approach developed in the present article is designed for operator learning in settings where nonlocal operators play a central role, such as in problems arising in plasma physics or computational neuroscience \cite{ANIE_NAT}. Since the general result in Theorem~\ref{thm:Universal} applies to arbitrary Banach spaces, we expect that it will be relevant to problems in the theory of PDEs and integral equations (IEs) formulated on Sobolev or H\"older spaces. The more specialized framework established in Theorem~\ref{thm:linear_universal}, though restricted to a particular case, is highly pertinent to machine learning applications, where mappings are typically learned with $L^p$ losses|namely, within $L^p$ spaces. Our approach is especially advantageous in scenarios where an explicit functional description of the approximation basis and its associated projection is required, as it might be the case in Galerkin-type and kernel methods for PDEs and IEs.

	This article is organized as follows. In Section~\ref{sec:Leray_Schauder} we provide the universal approximation result for general Banach spaces by means of Leray-Schauder mappings. In Section~\ref{sec:L_p} we consider $L^p$ spaces, and obtain a universal approximation result with linear projections on finite spaces of polynomials. In Section~\ref{sec:L_2} we consider the case $p=2$, and give some sufficient conditions for the universal approximation results to be applicable, along with examples. In Section~\ref{sec:Fixed_Point} we consider operator equations for operator learning problems. We determine some conditions under which our framework produces solutions to the projected equations that converge to solutions of the operator equation, reformulated as a fixed point problem. We conclude with some remarks, in Section~\ref{sec:future}, that describe future work on the algorithmic implementation of the deep learning methodology for operator learning based on this theoretical work.

	\section{Nonlinear Projections for Operator Learning in Banach spaces}\label{sec:Leray_Schauder}

	Let $X$ be a Banach space and let $K$ be a compact subset. 
	We recall here the construction of the maps of Leray and Schauder that were used in the proof of their celebrated fixed point theorem. Since $K$ is compact, for any choice of $\epsilon >0$ we can find a finite set  $\{x_i\}_{i=1}^n$ such that $K \subset \bigcup_{i=1}^n B(x_i,\epsilon)$, where $B(x_i,\epsilon)$ is the $\epsilon$-ball around $x_i$. We let $E$ be the span of the elements $x_i$. We define $P : K \longrightarrow E$ by the assignment
			\begin{eqnarray*}
				P x = \sum_{i=1}^{n}\frac{\mu_i(x) x_i}{\sum_{j=1}^n\mu_j(x)},
			\end{eqnarray*}
			where 
			\begin{eqnarray*}
			\mu_i(x) = \begin{cases}
			\epsilon - \|x-x_i\|, \hspace{1cm} \|x-x_i\| \leq \epsilon\\
			0, \hspace{3cm} \|x-x_i\| > \epsilon
			\end{cases}
			\end{eqnarray*}
			for all $i=1, \ldots , n$. 
	Then, $P$ is continuous and, moreover, for each $x\in K$ it holds that $\|x-P(x)\| < \epsilon$. 
	We will refer to these operators as {\it Leray-Schauder projections}, following the same convention as \cite{Topological}, although they are not linear. 
	
	If the elements $x_i$ are in $K$ and satisfy the property that $\|x_i-x_j\| \geq \epsilon$ for each $i\neq j$, it follows that $P(x_i) = x_i$ for each $i$. We want to show that it is always possible to choose $\{x_i\}$ with such property. In fact, let $x_1$ be any element of $K$. If $K \subset B(x_1,\epsilon)$, then there is nothing to prove. Otherwise, let $x_2\in K-B(x_1,\epsilon)$. Then, $\|x_1-x_2\|\geq \epsilon$ by construction. If $K\subset B(x_1,\epsilon)\cup B(x_2,\epsilon)$, then the process stops. Otherwise, we can choose $x_3$ in $K$ which has distance from $x_1$ and $x_2$ larger than or equal to $\epsilon$. So proceeding, we define a sequence $\{x_n\}$ in $K$. Suppose that the process does not stop after finitely many steps. Then, since $K$ is compact in a metric space, it is sequentially compact. Therefore, there is a subsequence $x_{n_k}$ which converges to $x\in K$. For a given $\epsilon' < \frac{\epsilon}{2}$, we find two indices $n_r$ and $n_{r+1}$, for $r$ large enough, such that $\|x_{n_r} - x_{n_{r+1}}\| \leq \|x_{n_r} - x\| + \|x - x_{n_{r+1}}\| < 2\epsilon' < \epsilon$, against the construction of $x_n$. This shows that there is a finite number $n$ such that $K \subset \bigcup_{i=1}^n B(x_i,\epsilon)$, and such that $\|x_i-x_j\| \geq \epsilon$ whenever $i\neq j$. Since $P(x_i) = x_i$, for a suitable choice of $\{x_i\}$, we can think of $P$ as being a ``projection''. However, this property is not used below in practice. 
	
	\begin{theorem}\label{thm:Universal_uniform}
		Let $X$ and $Y$ be Banach spaces, let $T: X\longrightarrow Y$ be a continuous (possibly nonlinear) map, and let $K$ be a compact subset of $X$, such that $T$ is uniformly continuous on an open neighborhood $U$ of $K$. Then, for any choice of $\epsilon > 0$ there exist natural numbers $n,m\in \mathbb N$, finite dimensional subspaces $E_n \subset X$ and $E_m \subset Y$, a  continuous map $P_n : K\longrightarrow E_n$ (defined and continuous on a neighborhood of $K$), and a neural network $f_{n,m} : \mathbb R^n \longrightarrow \mathbb R^m$ such that for every $x\in K$
		\begin{eqnarray}
		||T(x) - \phi_m^{-1}f_{n,m}\phi_nP_n(x)|| < \epsilon,
		\end{eqnarray}
		where $\phi_k : E_k \longrightarrow \mathbb R^k$ indicates an isomorphism between the finite dimensional space $E_k$ and $\mathbb R^k$.  
	\end{theorem}
	\begin{proof}
		Since $T$ is uniformly continuous on $U$, and $U\supset K$, we can choose $\delta>0$ such that $\|T(x_1) - T(x_2)\| < \frac{\epsilon}{3}$ whenever $\|x_1-x_2\| < \delta$, with $x_1,x_2\in U$, and such that $B(x,\delta) \subset U$ whenever $x\in K$.  
		 Corresponding to such choice of $\delta>0$, by compactness of $K$, we find a finite subset  $\{x_i\}_{i=1}^{n'}$ such that $K \subset \bigcup_{i=1}^{n'} B(x_i,\delta)$. 
		 Let $E_n$ be the subspace of $X$ spanned by $x_1, \ldots, x_{n'}$, where $n$ is the dimension of $E_n$. We can find a continuous map $P_n : K \longrightarrow E_n$ such that for every $x\in K$, we have
		\begin{eqnarray}\label{ineq:uniform}
		\|x-P_n(x)\| < \delta.
		\end{eqnarray}
		In particular, notice that $P_n(x)\in U$ for all $x\in K$. Since both $T$ and $P$ are continuous, the set $L = TP_n(K)$ is compact. We can then find a finite subset  $\{y_j\}_{i=1}^{m'}$ such that $L \subset \bigcup_{j=1}^{m'} B(y_j,\frac{\epsilon}{3})$. We define $E_m$ to be the linear span of $y_1, \ldots, y_{m'}$, where $m$ is the dimension of $E_m$.  Applying again the arguements of Leray-Schauder, there exists a continuous map $P_m : L \longrightarrow E_m$ satisfying the property that
		\begin{eqnarray}\label{ineq:Ler-Sch}
		\|x-P_m(x)\| < \frac{\epsilon}{3},
		\end{eqnarray}
		for any choice of $x$ in $L$. 
		We consider an isomorphism $\phi_n: E_n \longrightarrow \mathbb R^n$, obtained by choosing an arbitrary basis in $E_n$. Since any linear map between finite dimensional normed spaces is continuous, it follows that $\phi_n$ is continuous and has continuous inverse. Similarly, we can identify $E_m$ with $\mathbb R^m$ through a map $\phi_m$ which is continuous and has continuous inverse. 
		The map $T$ induces a continuous map $T_{n,m} : E_n \longrightarrow E_m$ which is obtained as $T_{n,m} = P_mT_{|E_n}$. Corresponding to such $T_{n,m}$, we introduce $F_{n,m}: \phi_n(E_n) \longrightarrow \mathbb R^m$ by filling the commutative diagram
		\begin{center}
			\begin{tikzcd}
				E_n\arrow[rr,"T_{n,m}"] & & E_m\arrow[d,"\phi_m"]\\
				 \mathbb R^n\arrow[u,"\phi_n^{-1}"]\arrow[rr,dashed,"F_{n,m}"] & & \mathbb R^m
			\end{tikzcd}.
		\end{center}
		In other words, we set $F_{n,m} := \phi_mT_{n,m}\phi_n^{-1} :  \mathbb R^n\longrightarrow \mathbb R^m$.
		The function $F_{n,m}$ is continuous, since it is a composition of continuous functions. As such, using the universal approximation properties of neural networks \cite{Horn,Pink,Fun}, we can find a neural network $f_{n,m} : \mathbb R^n \longrightarrow \mathbb R^m$ such that 
		\begin{eqnarray}\label{ineq:NN_approx}
		\|F_{n,m}(x) - f_{n,m}(x)\| < \frac{\epsilon}{3\|\phi_m^{-1}\|},
		\end{eqnarray}
		for all $x$ in the compact $\phi_nP_n(K) \subset \phi_n(E_n\cap U)$. 
		Finally, for all $x\in K$ we have
		\begin{eqnarray*}
			\|T(x) - (\phi_m^{-1}f_{n,m}\phi_n)P_n(x)\| &\leq& 
			\|T(x) - TP_n(x)\| + \|TP_n(x)- P_mTP_n(x)\| \\
			&& + \|P_mTP_n(x) -  (\phi_m^{-1}f_{n,m}\phi_n)P_n(x)\|\\
			&<&  \frac{\epsilon}{3} + \frac{\epsilon}{3} + \frac{\epsilon}{3}\\
			&=& \epsilon,
		\end{eqnarray*}
		where each term is seen to be smaller than $\frac{\epsilon}{3}$ as follows. We see that $\|T(x) - TP_n(x)\| < \frac{\epsilon}{3}$ holds because of \eqref{ineq:uniform}, the definition of uniform continuity, the fact that $P_n(x)\in U$ for each $x\in K$, and the choice of $\delta$.   
		We have  $\|TP_n(x)- P_mTP_n(x)\|< \frac{\epsilon}{3}$ because of \eqref{ineq:Ler-Sch} since $TP_n(x)\in L$ by definition. Moreover, $\|P_mTP_n(x) -  (\phi_m^{-1}f_{n,m}\phi_n)P_n(x)\| < \frac{\epsilon}{3}$ holds because of the choice of the neural network $f_{n,m}$, i.e. \eqref{ineq:NN_approx}, and the fact that $P_mT_{|E_n} = T_{n,m} = (\phi_m^{-1}F_{n,m}\phi_n)$ by construction. This completes the proof.
	\end{proof}

	It is possible to improve the preceding result by removing the assumption that $T$ be uniformly continuous in a neighborhood $U$ of $K$ as we now show. 

	\begin{theorem}\label{thm:Universal}
		Let $X$ and $Y$ be Banach spaces, let $T: X\longrightarrow Y$ be a continuous (possibly nonlinear) map, and let $K\subset X$ be a compact subset. Then, for any choice of $\epsilon > 0$ there exist natural numbers $n,m\in \mathbb N$, finite dimensional subspaces $E_n \subset X$ and $E_m \subset Y$, a continuous map $P_n : K\longrightarrow E_n$ (defined and continuous on a neighborhood of $K$), and a neural network $f_{n,m} : \mathbb R^n \longrightarrow \mathbb R^m$ such that for every $x\in K$
	\begin{eqnarray}
		||T(x) - \phi_m^{-1}f_{n,m}\phi_nP_n(x)|| < \epsilon,
		\end{eqnarray}
		where $\phi_k : E_k \longrightarrow \mathbb R^k$ indicates an isomorphism between the finite dimensional space $E_k$ and $\mathbb R^k$.  
	\end{theorem}
	\begin{proof}
				Let $\psi$ be an approximation of $T$ on the compact $K$ which is uniformly continuous on an open neighborhood $U$ of $K$, such that $\|T(x) - \psi(x)\| < \frac{\epsilon}{2}$ for all $x\in K$. This can be achieved as follows, for example. Let $\{B(x_i,\delta)\}_{i=1}^n$ be a finite cover of $K$ where $\delta$ is such that $\|T(y_1) - T(y_2)\| < \epsilon$ whenever $\|y_1-y_2\| < \delta$, whose existence is guaranteed by the fact that $T$ is uniformly continuous on $K$ by the Heine-Cantor theorem. It is known that given an open cover of a metric space, there exists a partition of unity of locally Lipschitz functions subordinate to the open cover \cite{luukkainen1976elements}. Applying this fact on the metric space $X$, with cover $V = \{B(x_i,\delta)\}_{i=1}^m\cup (X-K)$, we find a locally Lipschitz partition of unity $\{\psi_i\}_{i=1}^{n+1}$ subordinate to $V$. The map $\psi$ can be obtained by setting 
				\begin{eqnarray*}
						\psi(x) := \sum_{i=1}^n \psi_i(x)T(x_i) + \psi_{n+1}(x)y^*, 
				\end{eqnarray*}
				where $y^*$ is an arbitrarily fixed element of $Y$. 
				By construction, $\psi$ approximates $T$ on $K$ as required, and it is uniformly continuous on  a neighborhood $U$ of $K$ because it is locally Lipschitz and $K$ is compact. 
				Then, we can approximate $\psi$ on $K$ by $\phi_m^{-1}f_{n,m}\phi_nP_n$ as in Theorem~\ref{thm:Universal_uniform} with accuracy $\frac{\epsilon}{2}$. It follows that for all $x\in K$,  $\|T(x) - \phi_m^{-1}f_{n,m}\phi_nP_n(x)\| \leq  \|T(x) - \psi(x)\|+\|\psi(x) - \phi_m^{-1}f_{n,m}\phi_nP_n(x)\| < \epsilon$, as required. 
	\end{proof}
	
	\begin{corollary}
		With the assumption of Theorem~\ref{thm:Universal}, the neural network $f_{n,m}$ can be chosen to have a single hidden layer. 
	\end{corollary}
	\begin{proof}
		This fact follows from the proof of Theorem~\ref{thm:Universal_uniform}, with the observation that $f_{n,m}$ needs to approximate the function $F_{n,m}$ on a compact. Then, from the theory of neural network approximation, see \cite{Horn}, we can choose $f_{n,m}$ with a single hidden layer with the necessary properties. 
	\end{proof}

	We observe that Theorem~\ref{thm:Universal} holds for general Banach spaces $X$ and $Y$, whereas the universal approximation of \cite{chen}, which is the theoretical basis for \cite{DeepOnet}, holds for spaces of functions $Y$ with the uniform norm over compacts. The present result, therefore, substantially differs from that of \cite{chen}. 
	
	The Leray-Schauder projections present an implementational problem present in the theoretical framework of Theorem~\ref{thm:Universal}. Namely, we do not know how to choose the points used to obtain the spaces $E_n$ for the (nonlinear) projections. In a general Banach space, this might not be a trivial issue. In the rest of this article, we will develop a framework based on orthogonal multivariate polynomials to address this issue in the case of Banach spaces of functions $L^p_\mu$, and the Hilbert space $L^2_\mu$, with some measure $\mu$. An implementation of the Leray-Schauder mappings for deep learning has been given in \cite{leray_schauder}, based on the theoretical results of the present article.

	\section{Learning Linear Projections on Banach Spaces of functions}\label{sec:L_p}
	
	In this section we assume to work on the Banach space $L^p_\mu(S)$ where $\mu$ is some fixed finite Borel measure, and $S$ is a $\mu$-measurable subset of $\mathbb R^d$, which we will assume to be compact throughout. We will also assume that the measure is normalized, i.e. $\mu(S) = 1$, for simplicity since the reasoning can be adjusted to the case where $\mu(S) \neq 1$. A typical example would be the $L^p$ space on $[0,1]^{\times d}$ with Lebesgue measure $\mu$. In this article, by $L^p$ space we always mean the case where $1< p < +\infty$. 
	
	Let $\rho: S \longrightarrow \mathbb R$ be a $\mu$-integrable function in $L^q_\mu(S)$, where $\frac{1}{p} + \frac{1}{q} = 1$. Then, we say that $\rho$ is a weight function for $L^p_\mu(S)$. Let $\{p_k\}_{k=0}^r$ be a class of polynomials, then we say that they are orthogonal with respect to $\rho$ if 
	\begin{eqnarray*}
		\int_S p_kp_\ell \rho d\mu = \begin{cases}
			0 \hspace{2cm} k\neq \ell \\
			\neq 0 \hspace{1.6cm} k=\ell
		\end{cases}.
	\end{eqnarray*} 
	If the polynomials $p_k$ are normalized to $1$, i.e. $\int_S p_k^2 \rho d\mu = 1$, we will say that they are orthonormal.
	The weight function $\rho$, along with a class of orthogonal polynomials $\{p_k\}_{k=0}^r$ defines a projection on the subspace $E_r$ spanned by the polynomials through the functional defined according to the assignment
	\begin{eqnarray*}
		\mathcal L(f) = \int_S f(\mathbf x)\rho(\mathbf x) d\mu.
	\end{eqnarray*}
	The functional $\mathcal L$ defines a quasi-inner product.
	The projection on $E_r$ is then explicitly given by $P_r(f) = \sum_k \mathcal L(fp_k)\frac{p_k}{\mathcal L(p_k^2)}$. If the functional is positive (i.e. if $\rho$ has non-negative values), then we can normalize the polynomials by $\sqrt{\mathcal L(p_k^2)}$, and simply define $P_r(f) = \sum_k \mathcal L(fp_k)p_k$, where each $p_k$ is normalized to $1$. The maps $P_n$ are directly seen to be continuous. More generally, when $\mathcal L$ is an arbitrary functional such that 
	\begin{eqnarray*}
		\mathcal L(p_kp_\ell) = \begin{cases}
			0 \hspace{2cm} k\neq \ell \\
			\neq 0 \hspace{1.6cm} k=\ell
		\end{cases},
	\end{eqnarray*} 
	we will still say that $p_k$ are orthogonal (or orthonormal) polynomials with respect to $\mathcal L$, and in this case we can define a projection $P_n$ as before for all $n$.
	
	\begin{definition}\label{def:neural_proj}
		{\rm 
			A {\it neural projection operator} $\mathfrak S_{n,m,r}$ is a quadruple 
			$$(F_{n,m}, \rho_1, \rho_2, \{p^1_k\}_{k=0}^{r_1}, \{p^2_k\}_{k=0}^{r_2}),$$ 
			where $F_{n,m} : \mathbb R^n \longrightarrow \mathbb R^m$ is a neural network, $\rho_i: \mathbb R^d \longrightarrow \mathbb R$ are neural network weight functions, and $\{p^i_k\}_{k=0}^{r_i}$ is a set of orthogonal polynomials with respect to $\rho_i$.
		}
	\end{definition}
	This class of deep learning models consists of algorithms with a learnable projection over a multivariate orthogonal polynomial basis, along with a neural network that models a projected operator. We want to show now that neural projection operators are universal approximators of continuous (possibly nonlinear) operators in $L^p$ spaces. 
	
	The work of Kowalski \cite{Kow1,Kow2} and Xu \cite{Xu1,Xu2} has characterized the algebraic properties that the polynomials $\{p_k\}_{k=0}^\infty$ need to satisfy to be an orthogonal basis for the space of polynomials with respect to some functional $\mathcal L$. They have also shown that there are bases of such polynomials. In \cite{Kow1}, Kowalski has also given a condition for which the functional $\mathcal L$ is continuous on the space of polynomials in $\|\cdot\|_2$ norm. We assume that  $\{p_k\}_{k=0}^\infty$ is orthogonal with respect to a functional $\mathcal L$, continuous in $\|\cdot \|_p$ norm, and such that the projections $P_n$ are uniformly bounded in $n$. Also, we assume that no polynomial vanishes $\mu$-a.e. on $S$, e.g. as it happens if $S$ has nontrivial interior and $\mu$ has support containing an open subset of $S$. We want to show that given a basis and a continuous functional, we can approximate any continuous operator between $L^p$ spaces with a neural projection operator. 
	
	\begin{theorem}\label{thm:linear_universal}
		Let $T : L^{p_1}_\mu(S) \longrightarrow L^{p_2}_\mu(S)$ be a continuous (possibly nonlinear) operator, and let $X$ be a compact subset of $L^{p_1}_\mu(S)$. Let $\{p^i_k\}_{k=0}^\infty$ and $\mathcal L^i$ be as above, for $i=1,2$. Then, for any choice of $\epsilon >0$,  we can find a neural projection operator  $\mathfrak S_{n,m,r}$ such that 
		\begin{eqnarray}\label{ineq:linear_approx}
		||T(f) - \phi_m^{-1}f_{n,m}\phi_n\hat P_n(f)||_{p_2} < \epsilon,
		\end{eqnarray}
		for all $f\in X$, where $\phi_i$ are isomorphisms as before, and $\hat P_n$ is a learned continuous linear map.
	\end{theorem}
	\begin{proof}
		As in the case of Theorem~\ref{thm:Universal}, we can split the proof in two steps. One where we approximate a mapping which is uniformly continuous on a neighborhood $U$ of $X$, and then approximate the original map by such uniformly continuous map on $U$, using a locally Lipschitz partition of unity. Since the second step is fixed, and does not depend on the current setup, we can focus on the case where $T$ is uniformly continuous on $U \supset X$. 
		
		Given the fact that $\mathcal L$ is continuous by assumption on the space of polynomials in the $\|\cdot \|_p$ norm, where $p$ is either $p_1$ or $p_2$,  we can extend $\mathcal L$ to the whole $L^p_\mu(S)$ continuously by the continuous linear extension theorem, using the density of polynomials in $L^p_\mu(S)$. In fact, since $S$ is compact and $\mu$ is a finite measure, the space of continuous functions is dense in $L^p(S,\mu)$, and by the Stone-Weierstrass theorem, we have that polynomials are uniformly dense in the continuous functions. 
		Using the Riesz representation theorem for $L^p$ spaces, we can find a function $\rho$ in $L^q_\mu(S)$, where $\frac{1}{p} + \frac{1}{q} = 1$, such that $\mathcal L(f) = \int f\rho d\mu$ for all $f\in L^p_\mu(S)$. Moreover, $\|\mathcal L\| = \|\rho\|_q$. For any fixed number of polynomials $p_{i_1}, p_{i_2}, \ldots, p_{i_n}$ we have a projection $P_n : L^p_\mu(S) \longrightarrow E_n$ defined as $P_n(f) = \sum_{k=1}^n \mathcal L(fp_{i_k})\frac{p_{i_k}}{\mathcal L(p_{i_k}^2)}$, where $E_n$ indicates the span of the polynomials. 
		The projections $P^1_n$ and $P^2_n$, corresponding to the polynomials $\{p^1_i\}$ and $\{p^2_j\}$, respectively, will be shortened to $P_n$ for simplicity. We will denote by $\rho_i$ ($i=1,2$) the corresponding weight functions. For notational simplicity, $\|\rho_i\|$ will indicate $\|\rho_i\|_{q_i}$, $i=1,2$.
		
		Let $\delta >0$ be fixed. For simplicity, we just write $\|\cdot \|_p$ for the $p_1$-norms. We choose a finite set of functions that approximate any element of $X$ with accuracy $\epsilon'$, where $\epsilon'$ is sufficiently small so that $\epsilon'\|P_n\| < \frac{\delta}{3}$ for each $n$, consisting of polynomials $f_1, \ldots, f_d$, due to compactness of $X$ and density of polynomials in $L^p_\mu(S)$. We can then find $\ell$ polynomials $p_{i_1}, \ldots, p_{i_\ell}$ of $\{p_k\}_{k=0}^\infty$ such that $f_1, \ldots, f_d$ are in the linear span of $p_{i_1}, \ldots, p_{i_\ell}$. Upon possibly choosing a smaller $\delta$, we can also assume that $B(f,\delta) \subset U$ for each $f\in X$. Let $E_n$ be the span of $p_0, \ldots, p_n$ where $n> \max\{i_1, \ldots, i_\ell\}$. We let $P^1_n: L^{p_1}_\mu(S) \longrightarrow E_n$ denote the projection discussed above. Let now $M>0$ be larger than $\|f\|_p$ for all $f\in X$ (since $X$ is bounded such $M$ exists), and define $N = \sum_{k=1}^n \frac{\|p_k\|_\infty \|p_k\|_p}{|\mathcal L(p_k^2)|}$. Using the density of polynomials in $L^{q_1}_\mu(S)$, we can find a polynomial $\zeta_1$ such that $\|\zeta_1 - \rho_1\|_{q_1} < \frac{\delta}{6MN}$. Using the results on universal approximation of continuous functions by means of neural networks \cite{Pink,Fun,Horn}, we can find a neural network $\hat \rho_1$ such that $\|\hat \rho_1 - \zeta_1\|_\infty < \frac{\delta}{6MN}$ on the compact $S$. Therefore, we also have that $\|\hat \rho_1 - \rho_1\|_{q_1} < \frac{\delta}{3MN}$. Let us define the linear map $\hat P_n: L^{p_1}_\mu(S) \longrightarrow E_n$, defined by $\hat P_n(f) = \sum_{k=1}^n \int fp_k\hat \rho_1 d\mu \cdot \frac{p_k}{\mathcal L(p_k^2)}$. 
		We want to show that for all $f$ in $X$, we have $\|\hat P_n(f) - f\|_p < \delta$, and therefore that in particular $\hat P_n(X) \subset U$ . By construction, if $f\in X$, we can find a polynomial $f_r$, for some $r$ in $\{1, \ldots , d\}$, such that $\|f-f_r\|_p < \epsilon'$. Then, we have
		\begin{eqnarray*}
			\|\hat P_n(f) - f\|_p
			&\leq& \|\hat P_n(f) - P_n(f)\|_p + \|P_n(f) - P_n(f_r)\|_p + \| P_n(f_r) - f\|_p\\
			&=& \|\hat P_n(f) - P_n(f)\|_p + \|P_n(f-f_r)\|_p + \|f_r - f\|_p\\
			&\leq&  \sum_{k=1}^n|\int fp_k(\hat \rho-\rho)d\mu| \frac{\|p_k\|_p}{|\mathcal L(p_k^2)|} + \|P_n\| \|f - f_r\|_p + \epsilon'\\
			&\leq&  \sum_{k=1}^n \|f\|_p \|\hat \rho_1 - \rho_1\|_{q_1}\|p_k\|_\infty\frac{\|p_k\|_p}{|\mathcal L(p_k^2)|}  + \|P_n\|\|f - f_r\|_p + \epsilon'\\
			&<& \frac{\delta}{3} + \frac{\delta}{3} + \frac{\delta}{3}\\
			&=& \delta, 
		\end{eqnarray*}
		where we have used the fact that the functional $\psi(f) = \int f(\hat\rho - \rho)d\mu$ is continuous on $L^p$, and therefore $|\int f(\hat\rho - \rho)d\mu| \leq \|\hat\rho - \rho\|_q \|f\|_p$, along with the fact that $\|fp_k\|_p \leq \|f\|_p\|p_k\|_\infty$ for each $k$. 
		Similarly, one can construct $\hat P^2_m$ which maps onto an $m$-dimensional space $E_m$ such that $\|\hat P^2_m(f) - f\|_{p_2} < \frac{\epsilon}{3}$ whenever $f$ is in the compact  $Y = T(X)$. We can therefore now repeat the same construction of the proof of Theorem~\ref{thm:Universal}, using the proof of Theorem~\ref{thm:Universal_uniform}, where we use $\hat P^1_n$ and $\hat P^2_m$ instead of the Leray-Schauder (nonlinear) maps. 
		We will not repeat the details, as they are virtually identical to the combination of the proofs of Theorem~\ref{thm:Universal_uniform} and Theorem~\ref{thm:Universal}, but just observe that the proof did not need the fact that $x_i$ were chosen in the compact, but only that the Leray-Schauder maps sent the compact inside $U$, which was ensured in the present construction as well. 
		This will give us a neural network $f_{n,m}$ such that \eqref{ineq:linear_approx} is satisfied, therefore completing the proof.
	\end{proof}

	If $\sum_{k=0}^n\frac{\|p_k\|_\infty\|p_k\|_p}{|\mathcal L(p_k^2)|}$ is uniformly bounded in $n$, it would follow that the $P_n$ are uniformly bounded, since 
	\begin{eqnarray*}
		\|P_n(g)\|
		&=& \|\sum_{k=0}^n \int gp_k\rho d\mu \cdot \frac{p_k}{\mathcal L(p^2_k)}\|\\
		&\leq& \sum_{k=0}^n\int |gp_k\rho| d\mu \cdot \frac{\| p_k\|}{|\mathcal L(p^2_k)|}\\
		&\leq& \|g\|_p \|\rho\|_q \sum_{k=0}^n\|p_k\|_\infty  \frac{\| p_k\|}{|\mathcal L(p^2_k)|}.
	\end{eqnarray*}
	This is therefore an explicit bound that can be imposed in the training algorithm, if the polynomials $p_k$ are learned as well. 
	
	\begin{remark}
		{\rm 
			Theorem~\ref{thm:linear_universal} has the fundamental assumption that the functionals $\mathcal L^i$ determined by the polynomials $\{p^i_k\}_{k=0}^\infty$ are continuous in the $p$-norm on the space of polynomials. We will consider the important case $p=2$ where, applying results of Kowalski in \cite{Kow1}, we can have sufficient conditions for this to happen. In particular, ensuring such conditions in a deep learning algorithm would allow the approximation result of Theorem~\ref{thm:linear_universal} to be applicable.  
		}
	\end{remark}
	
	\section{Learning Linear Projections on the Hilbert space}\label{sec:L_2}
	
	We now consider the particularly important example of the Hilbert space $L^2([-1,1]^n)$, which is of great importance in applications, including deep learning. In fact, the loss function used in deep learning problems is often the mean squared error (MSE), which corresponds to a discretized version of $L^2$ norm. 
	
	 We recall the following condition considered by Kowalski in \cite{Kow1}. 
	
	\begin{hypothesis}[Kowalski]\label{hyp:Kow}
		For each $k = 0,1, \ldots$ there exist matrices $A_k, B_k, C_k$ such that 
		\begin{eqnarray*}
			{\rm rank} A_k = r_n^{k+1},
		\end{eqnarray*}
		where $r_n^{k+1}$ is the number of degree $k+1$ polynomials in $\{p_i\}$.  The recursion formula
		\begin{eqnarray*}
			\vec{xp_i} = A_i\vec p_i + B_i\vec p_i + C_i\vec p_{i-1}
		\end{eqnarray*}
		holds, where the symbol $\vec p_i$ represents vectors consisting of all polynomials in the family $\{p_k\}$ of degree $i$. For any arbitrary sequence of matrices such that $D_kA_k = \mathbb 1$ the recursion
		\begin{eqnarray*}
			I_0,\ \ \ I_{j+1} = D_j {\rm bp} (I_jC^T_{j+1}),
		\end{eqnarray*}
		gives positive definite matrices, where ${\rm bp}$ is the operation (defined in \cite{Kow1}) that performs block permutation. 
	\end{hypothesis}
	
	We have the following useful result, which is proved using the same approach of Theorem~2 in \cite{Kow1}.
	
	\begin{lemma}\label{lem:Kow}
		Let $\{p_k\}_{k=0}^\infty$ be a family of polynomials satisfying Hypothesis~\ref{hyp:Kow}. Let $\{q_k\}$ be an orthonormal polynomial complete sequence in $L^2([-1,1]^n)$, and suppose that 
		\begin{eqnarray}\label{eqn:p_q}
		q_k = \sum_{j=0}^\infty C_{kj}p_{j},
		\end{eqnarray}
		where $\sum_{k=0}^\infty C^2_{k1} < \infty$. Then, there exists a function $\rho \in L^2([-1,1]^n)$, and numbers $m_k \neq 0$, such that the functional $\mathcal L(f) := \int f\rho d\mu$ is continuous over $L^2([-1,1]^n)$ and it satisfies $\mathcal L(p_ip_k) = \delta_{ik}m_k$. 
	\end{lemma}
	
	The following is a special case of Theorem~\ref{thm:linear_universal} when $p=2$, by applying Lemma~\ref{lem:Kow}. Here we are assuming the same uniform boundedness condition of Theorem~\ref{thm:linear_universal}. 
	\begin{theorem}\label{thm:linear_universal_L2}
		Let $\{p_k\}_{k=0}^\infty$ be a family of polynomials as in Lemma~\ref{lem:Kow}, let $T:  L^2([-1,1]^n) \longrightarrow L^2([-1,1]^n)$ be a continuous operator, and let $X\subset L^2([-1,1]^n)$ be compact. Then, for any choice of $\epsilon >0$,  we can find a neural projection operator  $\mathfrak S_{n,m,r}$ such that 
		\begin{eqnarray}\label{ineq:linear_approx_L2}
		||T(x) - \phi_m^{-1}f_{n,m}\phi_n\hat P_n(x)||_2 < \epsilon,
		\end{eqnarray}
		for all $x\in X$, where $\phi_i$ are isomorphisms as before, and $\hat P_n$ is a learned continuous linear map.
	\end{theorem}
	\begin{proof}
		The main observation is that Lemma~\ref{lem:Kow} allows us to use the same proof as in Theorem~\ref{thm:linear_universal} since the functional $\mathcal L$ is continuous on $L^2([-1,1]^n)$. 
	\end{proof}
	
	As a particular case of the previous result one can take any complete sequence of orthonormal polynomials $\{p_k\}_{k=0}^\infty$ and set $q_k = p_k$ for all $k$, since in this case $C_{kj} = \delta_{kj}$ in Equation~\eqref{eqn:p_q}, and $\sum C_{k1} = C_{11} < \infty$, and where $\rho = 1$, so that $\mathcal L(fg) = \int fgd\mu = \langle f, g\rangle $. In this case, the uniform boundedness of the projections $P_n$ is automatic.

	\section{Approximations for Fixed Points}\label{sec:Fixed_Point}
	
	We now consider the problem of solving an equation of type
	\begin{eqnarray}\label{eqn:fixed_point}
	T(x) + f = x,
	\end{eqnarray}
	where $f$ is a fixed element of the Banach space $X$, and $T: X\longrightarrow X$ is an operator which is possibly nonlinear. The element $x\in X$ satisfying Equation~\eqref{eqn:fixed_point} is a fixed point for the operator. Our interest in such a problem stems from the fact that it is possible to frame operator learning tasks in terms of fixed point problems as in \cite{ANIE_NAT,Spectral,NIDE}. Our fundamental question in this section is whether it is possible to project Equation~\eqref{eqn:fixed_point} to a finite dimensional space, and if taking the limit $n \longrightarrow \infty$ one recovers a solution to Equation~\eqref{eqn:fixed_point}. If this is the case, then we can formulate an operator learning problem in terms of a neural projection operator, as in Definition~\ref{def:neural_proj}, and in the limit $n\longrightarrow +\infty$ for the size of the projection space used in the neural operator we recover solutions of the original system. Therefore, we would have the guarantee that upon taking a high enough projection dimension $n$ the learned operator would approximate $T$ and the solution would give an approximation to the real fixed point of $T$ in Equation~\eqref{eqn:fixed_point}. 
	
	We use the same framework as in Section~\ref{sec:L_p}, and we assume that the function $\rho\in L^q_\mu(S)$ (with $\frac{1}{p} + \frac{1}{q} = 1$) is given, which implies that the projection $P_n$ on $E_n$ it induces from $L^p_\mu(S)$ is continuous for all $n$. We assume that $\{p_k\}$ is such that the uniform boundedness of $P_n$ is ensured. 
	
	\begin{hypothesis}\label{hyp:Galerkin}
		We make the following assumptions. 
		\begin{enumerate}
			\item 
			The operator $T : L^p_\mu(S) \longrightarrow L^p_\mu(S)$ is completely continuous. 
			\item 
			The topological index of $T+f$ is nonzero, in the sense that $\deg(\mathbb 1-f-T,U,0) \neq 0$ in some open neighborhood $U$ of zero where $x = f + T(x)$ admits a unique solution.  
			\item 
			The operator $T$ is Fr\'echet differentiable. 
			\item 
			The value $1$ is not an eigenvalue of the Fr\'echet derivative of $T$ at $0$.
		\end{enumerate}
	\end{hypothesis}

			\begin{remark}
					{\rm 
							We point out that Urysohn operators represent a class of examples for the conditions in  Hypothesis~\ref{hyp:Galerkin} under relatively mild assumptions. A detailed study of the topological index for such operators is considered for example in \cite{Topological}. 
					}
			\end{remark}
	
	\begin{theorem}\label{thm:Galerkin}
		Under the assumptions (1)-(2) of Hypothesis~\ref{hyp:Galerkin}, for any choice of $n$ (large enough), the projected equation 
		\begin{eqnarray}
		T_n(x_n) + f_n = x_n,
		\end{eqnarray}
		where $T_n(y) = P_nT(y)$, and $f_n = P_nf$, admits a unique solution $x^*_n$. Moreover, $x^*_n \longrightarrow x^*$ where $x^*$ is a solution to Equation~\eqref{eqn:fixed_point}. If conditions (3)-(4) of Hypothesis~\ref{hyp:Galerkin} are also satisfied, then the rate of convergence of $x^*_n$ to $x^*$ is bounded by $(1+\epsilon_n)\|P_nx^*-x^*\|$ for a sequence $\epsilon_n $ converging to zero. 
	\end{theorem}
	\begin{proof}
		We apply the framework of \cite{Topological} on Galerkin's Method (Sections 3.3 and 3.4) to our construction, and consider the projected equation $x_n = f_n P_nTx_n$. Observe that in the setup considered in this article, the images of $P_n$ are nested and the union $\cup_{n\in \mathbb N} P_n(L^p_\mu(S))$ is dense in $L^p_\mu(S)$ because the polynomials used to construct $P_n$ are an algebraic basis for all polynomials, and $p < \infty$. Then, the uniform boundedness of $P_n$ implies that $P_nx$ converges to $x$ for each $x\in L^p_\mu(S)$. With $U$ as in Hypothesis~\ref{hyp:Galerkin} (2), then the projected compact perturbations have the same topological degree as the original map for all sufficiently large $n$, since $T(\overline U)$ is relatively compact, uniform boundedness along with pointwise convergence implies that $P_nT \longrightarrow T$ uniformly on $\overline U$. This ensures the existence of solutions $x^*_n$ of the projected equation, and convergence to a solution $x^*$ of Equation~\eqref{eqn:fixed_point} as in \cite{Topological}, Section 3.3. With the additional assumptions (3)-(4) of Hypothesis~\ref{hyp:Galerkin}, and indicating by $T'$ the Fr\'echet derivative of $T$ at zero, the bound on the convergence rate can be determined by writing the identity
		\begin{eqnarray*}
			\lefteqn{(\mathbb 1-T')(P_nx^*-x^*) + (T'-P_nT')(P_nx^*-x^*)}\\ 
			&=&
			(\mathbb 1 - T')(x^*_n-x^*) + (T'-P_nT')(x^*_n-x^*)\\ 
			&&+ P_n(T'P_n-T')(x^*_n-x^*) - P_n[Tx^*_n-Tx^*-T'(x^*_n-x^*)], 
		\end{eqnarray*}
		and applying the inverse of $\mathbb 1 - T'$ whose existence is guaranteed by the assumptions as in Section 3.4 of \cite{Topological}. The sequence $\epsilon_n$ depends on the norm of $\mathbb 1 - T'$, which therefore determines the rate of convergence. 
 	\end{proof}
	
	The rate of convergence of the approximated solutions to the solution of the original equation can be studied also using the methods of \cite{Atk-Pot}.
	There is an important case where we can apply Theorem~\ref{thm:Galerkin}. This is the case also considered in Section~\ref{sec:L_2} where $p=2$ (i.e. we have the Hilbert space), and we use an orthonormal basis $p_k$. Here $\rho = 1$ is just the identity function, and the projection is orthogonal. In this case it is known that $\{P_n\}$ is uniformly bounded.
	
	 If we have a sequence of $\{p_k\}$ such that $\sum_{k=0}^\infty\frac{\|p_k\|_\infty\|p_k\|_p}{|\mathcal L(p_k^2)|} < +\infty$, as mentioned after the proof of Theorem~\ref{thm:linear_universal}, it follows that the projections $P_n$ are uniformly bounded by the finite number $\|\rho\|_q\sum_{k=0}^\infty \frac{\|p_k\|_\infty\|p_k\|_p}{G_k(p_k)}$. Therefore we can apply the procedure above. This means that upon adding a constraint during training one can ensure numerical stability. 
	
	\section{Conclusions and Future Perspectives}\label{sec:future}  
	
	The results of this article provide a general theoretical foundation for approximating nonlinear operators between Banach spaces by reducing the problem to finite-dimensional approximation. The first main contribution is a universal approximation theorem for continuous operators on arbitrary Banach spaces over compact subsets. This is achieved through the construction of Leray--Schauder type finite-dimensional approximations, which allow one to replace the original operator by a composition consisting of a continuous projection onto a finite-dimensional subspace, a finite-dimensional neural network, and an embedding back into the target space. In this way, the approximation of nonlinear operators in infinite-dimensional Banach spaces is connected directly to the classical universal approximation properties of neural networks in finite dimensions.
	
	A second contribution of the paper is the development of a more concrete framework for operator learning on spaces of functions, especially on $L^p$ spaces. In this setting, the abstract finite-dimensional reductions are realized through projections onto polynomial bases, together with a learned finite-dimensional map between the corresponding coefficient spaces. Theorem~\ref{thm:linear_universal} shows that, under the stated assumptions, this construction yields a universal approximation result for continuous operators between $L^p$ spaces. In the Hilbert space case $p=2$, the orthogonality structure gives additional sufficient conditions ensuring that the required continuity properties are satisfied. Finally, the Galerkin-type result in Theorem~\ref{thm:Galerkin} shows that the projected fixed point equations are not only well posed at the finite-dimensional level, but also that their solutions converge to a solution of the original operator equation under Hypothesis~\ref{hyp:Galerkin}. Thus, the paper establishes both an approximation-theoretic result and a convergence result for the projected operator equations that arise from the proposed learning framework.

	We now conclude with a few remarks on future work based on the results of this paper. More specifically, we describe the algorithmic perspectives that follow from our present results. The approach concerns learning operators between $L^p$ spaces in a suitable projected space. We want to both learn projections $P_n$ and $P_m$ (possibly coinciding) via learning a basis of polynomials in order to apply Theorem~\ref{thm:linear_universal}. To do so, we need the polynomials to be orthogonal in the sense of Section~\ref{sec:L_2}. This can be done using the algebraic characterization of Kowalski \cite{Kow1,Kow2}, and Xu \cite{Xu1,Xu2}. These approaches generalize the well known one-dimensional case of recursion formulas of the Favard's Theorem. Therefore, one can recursively construct a family of orthogonal polynomials. 
	
	In addition, we need to learn a mapping between the projected spaces that approximates a given operator. Our objective is to learn this operator in such a way that its solutions approximate the solutions of a projected operator equation as in Equation~\ref{eqn:fixed_point}, which is a fixed point problem. A similar type of approach for specific projections on Chebyshev polynomials (spectral methods) and integral nonlinear operators has been pursued in \cite{Spectral}. The approach described in this article is more general, and we expect that it is more widely applicable in practice. 
	
	The results obtained in this article show that the methodology described produces a universal approximator under mild additional assumptions. Moreover, as discussed in Section~\ref{sec:L_2}, in the fundamental case of $p=2$ there are some direct conditions that can be imposed guaranteeing that the continuity of the functional holds, therefore giving a direct method for obtaining universal approximators. Additionally, Section~\ref{sec:Fixed_Point} shows that upon increasing the dimension of the projections the solutions obtained in the projected spaces converge to the solutions of the operator equation that is being modeled. Therefore, the model has some good convergence properties under the framework of Hypothesis~\ref{hyp:Galerkin}.
	
	Guaranteeing that the needed assumptions are satisfied during the learning process is an interesting computational problem, and leveraging the theoretical framework described in this article in practice is of its own interest in machine learning. 
	
	%

\end{document}